\documentclass{article}
\usepackage[utf8]{inputenc}
\usepackage{amsmath}
\usepackage{amssymb}
\usepackage{amsthm}

\newcommand{\mm}{\mathfrak{m}}
\newcommand{\qq}{\mathfrak{q}}
\newcommand{\bb}{\mathfrak{b}}
\newcommand{\afrak}{\mathfrak{a}}

\newcommand{\height}{\operatorname{ht}}
\newcommand{\Min}{\operatorname{AssMin}}
\newcommand{\rad}{\operatorname{rad}}
\newcommand{\ellnew}{l}

\newtheorem*{main}{Main Theorem}

\newtheorem{theorem}{Theorem}
\newtheorem{corollary}{Corollary}
\newtheorem{claim}{Claim}

\newtheorem{lemma}{Lemma}
\theoremstyle{remark}
\newtheorem{example}{Example}
\theoremstyle{remark}
\newtheorem{remark}{Remark}

\title{On the effective reduction of an ideal}
\author{Tomasz Rodak, Adam Różycki, Stanisław Spodzieja}
\date{22 June 2021}

\begin{document}
\maketitle

\abstract{It is well known that in the Noetherian local ring with
  infinite residue field the reduction of $\mm$-primary ideal may
  be given in the form of a sufficiently general linear combination of
  its generators. 
  In the paper we give a condition for the existence of such
  reduction in terms of the sum of degrees of the ideal fiber cone prime
  divisors in the case of any
  Noetherian local ring.}

\section{Introduction}

Let $R$ be a commutative ring with unity and let $\afrak$ be its ideal. We
say that ideal $\bb\subset\afrak$ is a \emph{reduction} of $\afrak$ 
if there exists positive integer $k$ such that
\begin{equation*}
  \afrak^{k+1}=\bb\afrak^k.
\end{equation*}
Obviously, ideal $\afrak$ is always a reduction of $\afrak$. In
practice it is often important to know if an ideal has a reduction which is in
some sens smaller then the ideal itself.

Assume now that $(R,\mm)$ is a Noetherian local ring. The \emph{fiber
  cone} of ideal $\afrak\subset R$ is the ring
\begin{equation*}
  \mathcal{F}_\afrak(R)=\frac{R}{\mm}\oplus\frac{\afrak}{\mm \afrak}\oplus\frac{\afrak^2}{\mm \afrak^2}\oplus\cdots
\end{equation*}
The ring $\mathcal{F}_\afrak(R)$ is a finitely generated graded ring
over the residue field $R/\mathfrak{m}$ by elements of degree one.
The Krull dimension of $\mathcal{F}_\afrak(R)$ is called \emph{analytic
  spread} of $\afrak$ and is denoted $\ell(\afrak)$.
It is well known that the minimal number of generators in a reduction
of $\afrak$ is bounded below by $\ell(\afrak)$ (see \cite[Corollary
8.2.5]{huneke}). Moreover, if the residue field $R/\mm$ is infinite,
then $\afrak$ has a reduction generated by $\ell(\afrak)$ elements and
every such reduction is minimal with respect to inclusion (see
\cite[Proposition 8.3.7]{huneke}). From \cite[Corollary 8.3.9]{huneke}
we have $\height\afrak\leqslant\ell(\afrak)\leqslant\dim R$. Hence, if $\afrak$
is a primary ideal belonging to $\mm$, then $\ell(\afrak)$ is equal to
the Krull dimension $\dim R$ of $R$. The following theorem is well
known (see \cite[Theorem 14.14]{matsumura}).

\begin{theorem}
  Let $(R,\mathfrak{m})$ be an $n$-dimensional Noetherian local ring,
  and suppose that $R/\mathfrak{m}$ is an infinite field; let
  $\qq=(u_1,\ldots,u_m)R$ be an $\mm$-primary ideal. Then if
  $y_i=\sum a_{ij}u_j$ for $1\leqslant i\leqslant n$ are $n$
  'sufficiently general' linear combinations of $u_1,\ldots,u_m$, the
  ideal $\bb=(y_1,\ldots,y_n)R$ is a reduction of $\qq$ and
  $y_1,\ldots,y_n$ is a system of parameters of $R$.
\end{theorem}


In this paper
we give an estimation for the number of linearly independent linear
forms which provide a reduction in the spirit of the above theorem for
any $\mm$-primary ideal with sufficiently small sum of the degrees of
prime divisors.

Before we state our main result, we give a definition.

Let $k$ be a field and let $R$ be a Noetherian graded $k$-algebra
finitely generated over $k$ by elements of degree one. Then we set
\begin{equation*}
  \operatorname{deg.rad}R:=\sum_{\mathfrak{p}\in\Min R}\deg R/\mathfrak{p},
\end{equation*}
where $\Min R$ denotes the set of minimal prime ideals in $R$.

The number $\operatorname{deg.rad}R$ is an algebraic counterpart of
the following notion introduced in the setting of complex algebraic
geometry (see \cite[VII.11.8]{lojasBook}). Let 
$V\subset\mathbb{P}^m$ be an algebraic set 
in a complex projective space $\mathbb{P}^m$. If we set
$\delta(V)=\sum\deg V_i$, where $V=\bigcup V_i$ is a decomposition into
irreducible components, then $\delta(V)=\operatorname{deg.rad}R$,
where $R=k[V]$.

Let $\qq=(u_1,\ldots,u_m)R$, where $(R, \mm)$ is a Noetherian local
ring. Then
\begin{equation}
  \mathcal{F}_\qq(R)\simeq\frac{\frac{R}{\mm}[X_1,\ldots,X_m]}{Q},\label{eq:isomorphism_Fq}
\end{equation}
 where $Q$ is the kernel of the
 homomorphism $\frac{R}{\mm}[X_1,\ldots,X_m]\to\mathcal{F}_\qq(R)$
 given by $X_i\mapsto u_i + \mm\qq$.
We will write $\bar P(X)$ for
the image in $\frac{R}{\mm}[X]$ of a polynomial $P(X)\in R[X]$. Observe, that
for a homogeneous polynomial $P(X)\in R[X]$ of degree $k$ we have
$\bar P(X)\in Q$ if and only if $\bar P(X)$ is a null-form of $\qq$,
i.e. $P(u_1,\ldots,u_m)\in\mm\qq^k$.

Let $k$ be a field. We will say that a sequence of linear forms
$N_1,\ldots,N_s\in k[X_1,\ldots,X_m]$, $s\geqslant m$ is independent if 
for any $1\leqslant i_1<\cdots<i_m\leqslant s$ the system
$N_{i_1},\ldots,N_{i_m}$ is linearly independent.

Our main result is

\begin{main}\label{thm:main_theorem}
  Let  $m$, $n$, $d$ be some positive integers, where $m\geqslant n$.
  Set $\ellnew=d(m-n)+n$. Let $(R,\mm)$ be a local Noetherian $n$-dimensional ring and let
  $\bar N_j\in \frac{R}{\mm}[X_1,\ldots,X_m]$, $j=1,\ldots,\ellnew$ be a sequence of independent linear forms. 
  For any $\mm$-primary ideal $\qq=(u_1,\ldots,u_m)R=(u)R$, such that
  \begin{equation*}
  \operatorname{deg.rad}\mathcal{F}_\qq(R)\leqslant d.
  \end{equation*}
  there exist $1\leqslant i_1<\cdots<i_n\leqslant\ellnew$, such that
  $\mathfrak{b}=(N_{i_1}(u),\ldots,N_{i_n}(u))R$ is a reduction of
  $\qq$ and $N_{i_1}(u),\ldots,N_{i_n}(u)$ is a system of parameters
  of $R$.
\end{main}


\section{Proof of the main result}

\begin{lemma}\label{lem:minimal_primes}
Let $I$ be an ideal in a Noetherian ring $R$ and let $\Min I=\{P_1,\ldots,P_s\}$.
If for $a\in R$ we set
$\Min(P_i+Ra)=\{P_{i1},\ldots,P_{is_i}\}$, then
\begin{equation*}
  \rad(I+Ra)=\bigcap_{i=1}^s\bigcap_{j=1}^{s_i}P_{ij}.
\end{equation*}
\end{lemma}

\begin{proof}
  Let $P$ be a minimal prime ideal over $I+Ra$. There exists $P_i$ such that $P_i\subset P$. Since $I+Ra\subset P_i+Ra\subset P$, then $P$ is minimal over $P_i+Ra$. This gives $\supset$.

  Now, take $P_{ij}$ which is minimal with respect to inclusion in the set $\{P_{ij}:i=1,\ldots,s,j=1,\ldots,s_i\}$.
  It is enough to prove that $P_{ij}$ is minimal over $I+Ra$. Suppose
  to the contrary, that there exists a prime ideal $P'\subset R$ such
  that $I+Ra\subset P'\subsetneq P_{ij}$. Then, there exists
  $P_{i'}\subset P'$ and since $a\in P'$ then as a result there exists
  $P_{i'j'}\subset P'$. Hence $P_{i'j'}\subsetneq P_{ij}$ which is
  impossible by the minimality of $P_{ij}$. 
\end{proof}

In what follows we will assume that $X=(X_1,\ldots,X_m)$ is a system
of variables.

\begin{lemma}\label{lem:rad_degree}
  Let $Q$ be a homogeneous
  ideal in a polynomial ring $k[X]$ over a field $k$. If $N\in k[X]$
  is a linear form, then  
  \begin{equation*}
    \operatorname{deg.rad}(k[X]/(Q + k[X]N)\leqslant
    \operatorname{deg.rad}(k[X]/Q).
  \end{equation*}
\end{lemma}

\begin{proof}
  Let $\Min Q=\{P_1, \ldots,P_s\}$ and let
  $\Min(P_i+k[X]N)=\{P_{i1},\ldots,P_{is_i}\}$. Observe that if
  $N\notin P_i$ then for any $j$, $\height P_{ij}=\height P_i + 1$. Indeed, let
  $S=k[X]_{P_{ij}}/P_ik[X]_{P_{ij}}$. Then
  \begin{equation*}
    \height P_{ij}=\dim
    k[X]_{P_{ij}} = \height P_ik[X]_{P_{ij}} + \dim S = \height P_i
    +\dim S. 
  \end{equation*}
  By Krull's
  principal ideal theorem $\height P_{ij}S\leqslant 1$, since $P_{ij}S$ is minimal over $(P_i+k[X]N)S$,
  which in turn is generated in $S$ by the image of $N$. The ring $S$ is
  an integral domain and the image of $N$ in $S$ is non-zero, hence
  $\dim S=\height P_{ij}S=1$, which proves the claim.
  
  From the above we have
  \begin{equation*}
    \deg (k[X]/(P_i + k[X]N))=\sum_{j=1}^{s_i}\lambda_j\deg
    (k[X]/P_{ij}),
  \end{equation*} 
  where $\lambda_j$ stands for the length of the ring
  $k[X]_{P_{ij}}/(P_i + k[X]N)k[X]_{P_{ij}}$ as a module over itself.
  On the other hand, since $\deg N=1$, then by
  \cite[Proposition 5.3.6]{GreuelPfisterSICP} we have
  \begin{equation*}
    \deg (k[X]/P_i)=\deg (k[X]/(P_i+k[X]N)).
  \end{equation*} 
  This and Lemma \ref{lem:minimal_primes} gives
  \begin{align*}
    \operatorname{deg.rad}(k[X]/Q)&=\sum_{i=1}^s\deg (k[X]/P_i)\\
    &=\sum_{i=1}^s\deg (k[X]/(P_i + k[X]N))\\
    &\geqslant\sum_{i=1}^s\sum_{j=1}^{s_i}\deg (k[X]/P_{ij})\\
    &\geqslant\operatorname{deg.rad}(k[X]/(Q + k[X]N)).
  \end{align*}
\end{proof}

\begin{lemma}\label{lem:opuszczenie_wymiaru}
Let $m$, $n$, $d$ be positive integers with
$m\geqslant n$, let $\ellnew=d(m-n)+1$, and let $N_j$, $j=1,\ldots,\ellnew$
be a system of independent linear functions. Then for any homogeneous
ideal $Q$ in $k[X]$, with $\dim (k[X]/Q)\leqslant n$ and
$\operatorname{deg.rad}(k[X]/Q)\leqslant d$, there exists $j$ such that
$\dim(k[X]/(Q + N_jk[X]))<n$. 
\end{lemma}

\begin{proof}
Let $\{P_1,\ldots,P_s\}$ be the set of minimal prime ideals of $Q$. We
have $d\geqslant\sum_i\deg (k[X]/P_i)\geqslant s$. It is enough to proof that
there exists $j\in\{1,\ldots,\ellnew\}$ such that $N_j\notin P_i$ for any $P_i$ with $\height
P_i=m-n$. Indeed, in this case for any prime ideal $P$ containing
$Q+N_jk[X]$ we have $\height P > \height P_i=m-n$, and thus $\dim(k[X]/(Q + N_jk[X]))<n$.
Now, suppose to the contrary, that for any
$j\in\{1,\ldots,\ellnew\}$ there exists $\psi(j)\in\{1,\ldots,s\}$ such
that $N_j\in P_{\psi(j)}$ and 
$\height P_{\psi(j)}=m-n$. Since $N_j$ are independent then ideal
generated by any $r\leqslant m$ of them is of height $r$, so we
must have $\#\psi^{-1}(i)\leqslant m-n$ for any $i$. Hence
\begin{equation*}
    d(m-n)+1=\ellnew=\#\psi^{-1}(\{1,\ldots,s\})\leqslant s(m-n)\leqslant d(m-n), 
  \end{equation*}
  which gives a contradiction. 
\end{proof}

\begin{lemma}\label{lem:main_lemma}
Let $m$, $n$, $d$ be positive integers with $m\geqslant n$, let $\ellnew=d(m-n)+n$, and let $N_j$, $j=1,\ldots,\ellnew$ be a system of independent linear functions. Then for any homogeneous ideal $Q$ in $k[X]$, with $\dim (k[X]/Q)\leqslant n$ and $\operatorname{deg.rad}(k[X]/Q)\leqslant d$, there exists $1\leqslant i_1\leqslant\cdots\leqslant i_n\leqslant\ellnew$ such that the ideal $Q + (N_{i_1},\ldots,N_{i_n})k[X]$ is primary ideal belonging to $(X_1,\ldots,X_m)k[X]$.
\end{lemma}

\begin{proof}
The proof will be carried out by induction with respect to $n\in\{1,\ldots,m\}$.
For $n=1$ the assertion
follows from Lemma \ref{lem:opuszczenie_wymiaru}. Indeed, there exists $N_{i_1}$ such that
$\dim(k[X]/(Q+N_{i_1}k[X]))=0$, hence
$Q+N_{i_1}k[X]$ is $(X_1,\ldots,X_m)k[X]$-primary. 

Assume that the assertion holds $1\leqslant n-1<m$. Take any homogeneous ideal $Q\subset k[X]$ with $\dim(k[X]/Q)\leqslant n$ and $\operatorname{deg.rad}(k[X]/Q)\leqslant d$. By Lemma \ref{lem:opuszczenie_wymiaru} there exists, say $N_1$, such that $\dim(k[X]/(Q+N_1k[X]))<n$. By Lemma \ref{lem:rad_degree}, $\operatorname{deg.rad}(k[X]/(Q+N_1k[X])\leqslant d$.
So the induction hypothesis is fulfilled for the image of the ideal $Q$ in $k[X]/k[X]N_1$. Hence there exists $N_{i_2},\ldots,N_{i_n}$ in the family
$N_2,\ldots,N_\ellnew$ of $\ellnew-1=d(m-n)+n-1=d((m-1)-(n-1))+n-1$ independent linear functions such that
$Q+(N_1,N_{i_2},\ldots,N_{i_n})$ is a primary ideal belonging to $(X_1,\ldots,X_m)k[X]$.
\end{proof}

Our aim is to prove

\begin{main}\label{thm:main_theorem}
  Let  $m$, $n$, $d$ be some positive integers, where $m\geqslant n$.
  Set $\ellnew=d(m-n)+n$. Let $(R,\mm)$ be a local Noetherian $n$-dimensional ring and let
  $\bar N_j\in \frac{R}{\mm}[X_1,\ldots,X_m]$, $j=1,\ldots,\ellnew$ be a sequence of independent linear forms. 
  For any $\mm$-primary ideal $\qq=(u_1,\ldots,u_m)R=(u)R$, such that
  \begin{equation*}
  \operatorname{deg.rad}\mathcal{F}_\qq(R)\leqslant d.
  \end{equation*}
  there exist $1\leqslant i_1<\cdots<i_n\leqslant\ellnew$, such that
  $\mathfrak{b}=(N_{i_1}(u),\ldots,N_{i_n}(u))R$ is a reduction of
  $\qq$ and $N_{i_1}(u),\ldots,N_{i_n}(u)$ is a system of parameters
  of $R$.
\end{main}

\begin{proof}
  From Step 2 in the proof
  of \cite[Theorem 14.14]{matsumura} we have
  \begin{claim}\label{claim}
    Let $\bb$ be the ideal of $R$ generated by $n$ linear combinations
    $L_i(u)=\sum a_{ij}u_j$ (for $1\leqslant i\leqslant n$) of
    $u_1,\ldots,u_m$ with coefficients in $R$. Then if we set
    $l_i(X)=\bar L_i(X)$, a necessary and sufficient condition for
    $\bb$ to be a reduction of $\qq$ is that the ideal
    $(Q,l_1,\ldots,l_n)$ of $\frac{R}{\mm}[X]$ is $(X_1,\ldots,X_m)$-primary.
  \end{claim}
  Let $Q\subset \frac{R}{\mm}[X]$ be a homogeneous ideal such that
  \eqref{eq:isomorphism_Fq} holds. 
  By Lemma \ref{lem:main_lemma} there exist $1\leqslant
  i_1<\cdots<i_n\leqslant\ellnew$ 
  such that $Q+(\bar N_{i_1},\ldots,\bar N_{i_n})\frac{R}{\mm}[X]$ is
  a primary ideal belonging to
  $(X_1,\ldots,X_m)\frac{R}{\mm}[X]$. Thus by the Claim \ref{claim} we
  get the assertion.  
\end{proof}

\begin{remark}
  If $\frac{R}{\mm}$ is infinite then by the above theorem any
  $\mm$-primary ideal has a reduction generated by $n$ linear
  combinations of its generators with the coefficients in $R$. In
  fact, from \cite[Theorem 14.14]{matsumura} it follows that this
  reduction may be given by a generic linear combinations.

  On the other hand, if $k=\frac{R}{\mm}$ is a finite field, then there is only
  finitely many independent linear forms in
  $k[X_1,\ldots,X_m]$. For example, if $k=\mathbb{Z}_2$ then there are
  $m+1$ independent linear forms: $X_1,\ldots,X_m$ and
  $X_1+\cdots+X_m$. 
\end{remark}

\begin{corollary}
  Let $(R,\mm)$ be a local Noetherian
  $n$-dimensional ring and let $\qq=(u_1,\ldots,u_{m})R$ be an
  $\mm$-primary ideal such that
  $\operatorname{deg.rad}\mathcal{F}_\qq(R)=1$. Then there exist
  $1\leqslant i_1<\cdots<i_n\leqslant m$ such that
  $(u_{i_1},\ldots,u_{i_n})R$ is a reduction of $\qq$.
\end{corollary}

\begin{corollary}
  Let $(R,\mm)$ be a local Noetherian
  $n$-dimensional ring and let $\qq=(u_1,\ldots,u_{n+1})R$ be an
  $\mm$-primary ideal such that
  $\operatorname{deg.rad}\mathcal{F}_\qq(R)\leqslant 2$. Then
  \begin{enumerate}
  \item $(u_{i_1},\ldots,u_{i_n})R$ for some $1\leqslant
    i_1<\cdots<i_n\leqslant n+1$ or
  \item $(u_{i_1},\ldots,u_{i_{n-1}},\sum_ju_j)R$ for some $1\leqslant
  i_1<\cdots<i_{n-1}\leqslant n+1$
  \end{enumerate}
  is a reduction of $\qq$.
\end{corollary}

\begin{corollary}
  Let  $m$, $n$, $d$ be some positive integers, where $m\geqslant n$.
  Set $\ellnew=d(m-n)+n$. Let $(R,\mm)$ be a local Noetherian
  $n$-dimensional ring such that $\#\frac{R}{\mm}\geqslant
  l$. Then there exist linear forms $N_j\in R[X_1,\ldots,X_m]$,
  $j=1,\ldots,\ellnew$  such that 
  $\bar N_j\in \frac{R}{\mm}[X_1,\ldots,X_m]$
  are independent. Hence, for any $\mm$-primary ideal 
  $\qq=(u_1,\ldots,u_m)R=(u)R$, such that 
  \begin{equation*}
    \operatorname{deg.rad}\mathcal{F}_\qq(R)\leqslant d.
  \end{equation*}
  there exist $1\leqslant i_1<\cdots<i_n\leqslant\ellnew$, such that
  $\mathfrak{b}=(N_{i_1}(u),\ldots,N_{i_n}(u))R$ is a reduction of
  $\qq$ and $N_{i_1}(u),\ldots,N_{i_n}(u)$ is a system of parameters
  of $R$.
\end{corollary}

\begin{proof}
  Set $\bar N_j(X):=X_1+\bar a_jX_2+\cdots+\bar a_j^{m-1}X_m$, where
  $a_j\in R$, $j=1,\ldots,l$ are any elements in $R$ such that $\bar
  a_j$ are pairwise distinct. Then $\bar N_j(X)\in\frac{R}{\mm}[X]$,
  $j=1,\ldots,l$ is a sequence of independent linear forms. The second
  part of the assertion follows from Theorem \ref{thm:main_theorem}.
\end{proof}

\begin{corollary}\label{cor:minimum_of_mult}
  Under the assumptions and notations of the Main Theorem we have
  \begin{equation*}
    e(\qq)=\min_{1\leqslant i_1<\cdots <i_n\leqslant\ellnew}
    \{e(\mathfrak{b}):\mathfrak{b}=(N_{i_1}(u),\ldots,N_{i_n}(u))R\},
  \end{equation*}
  where we set $e(\bb)=\infty$ if $\bb$ is not $\mm$-primary.
  What is more, if $R$ is formally equidimensional and $e(\qq)=e(\mathfrak{b})$ for some $\mathfrak{b}$ as above
  then $\mathfrak{b}$ is a reduction of $\qq$.
\end{corollary}

\begin{proof}
  For any $\mathfrak{b}=(N_{i_1}(u),\ldots,N_{i_n}(u))R$, where
  $1\leqslant i_1<\cdots <i_n\leqslant\ellnew$ we have $\mathfrak{b}\subset\qq$, hence
  $e(\qq)\leqslant e(\mathfrak{b})$. On the other hand, if $\mathfrak{b}$ is
  a reduction of $\qq$ then the integral closures of $\bb$ and $\qq$
  are equal (\cite[Corollary 1.2.5]{huneke}), hence
  $e(\qq)=e(\mathfrak{b})$ (\cite[Proposition 11.2.1]{huneke}). This
  gives the first part 
  of the assertion. The second follows from the Rees Theorem (see
  \cite[Theorem 11.3.1]{huneke}). 
\end{proof}

\begin{example}
  Let $\qq=(x^a, xy, y^b)R$, $2\leqslant a < b$ be an ideal in a power
  series ring $R:=k[[x,y]]$ over a field $k$. Let $Q\in k[X,Y,Z]$ be the
  ideal of null-forms of $\qq$. Then $Q$ is a principal ideal
  generated by $XZ$. Indeed, let $P(X,Y,Z)=XZ-x^{a-2}y^{b-2}Y^2\in
  R[X,Y,Z]$. Then $P(x^a,xy,y^b)=0$ and $\bar P(X,Y,Z)=XZ$, hence
  $XZ\in Q$. On the other hand, let $u_i=X^{r-i}Y^i$, $i=0,\ldots,r$,
  $v_j=Y^{r-j}Z^j$, $j=1,\ldots,r$ be all the monomials in $X,Y,Z$
  of degree $r>0$ which are not divisible by $XZ$. Then
  the images $u_i(x^a,xy,y^b)=x^{a(r-i)+i}y^i$, $v_j(x^a,xy,y^b)=x^{r-j}y^{(b-1)j+r}$ belongs to
  $\qq^r\setminus{\mm\qq^r}$ for any $i,j$. Moreover those images are
  pairwise disjoint. On the other hand, if $w$
  is a monomial of degree $r$ in $X,Y,Z$ divisible by $XZ$, then $w(x^a,xy,y^b)\in
  \mm\qq^r$. Thus, if $P(X,Y,Z)\in R[X,Y,Z]$ is a
  homogeneous form of degree $r$ such that $P(x^a,xy,y^b)\in\mm\qq^r$,
  then $\bar P(X,Y,Z)$ is divisible by $XZ$. This proves that
  $Q=(XZ)k[X,Y,Z]$ and as a result
  \begin{equation}
    \label{eq:2}
    \mathcal{F}_\qq(R)\simeq \frac{k[X,Y,Z]}{(XZ)k[X,Y,Z]}.
  \end{equation}
  Hence $\operatorname{deg.rad}\mathcal{F}_\qq(R)=2$ and $\ellnew=4$. Let
  \begin{align*}
    \bar N_1(X,Y,Z)&=X,\\
    \bar N_2(X,Y,Z)&=Y,\\
    \bar N_3(X,Y,Z)&=Z,\\
    \bar N_4(X,Y,Z)&=X+Y+Z.
  \end{align*}
  Now, by Corollary \ref{cor:minimum_of_mult} between $\binom{4}{2}=6$ possible choices for $N_{i_1},N_{i_2}$ there exists
  the only one $N_2,N_4$ which gives the reduction of $\qq$. Indeed,
  \begin{equation*}
    e((N_2(x^a,xy,y^b),N_4(x^a,xy,y^b))R)=e((xy,x^a+xy+y^b)R)=a+b,
  \end{equation*}
  but for the other choices the corresponding ideal
  is not $\mm$-primary (for $(X,Y)$ and ($Y,Z)$) or has multiplicity $ab$
  (for $(X,Z)$, $(X,X+Y+Z)$ and $(Z,X+Y+Z)$).
\end{example}

\begin{example}
  Let $0<a<b<c$ be integers such that $c>r(b-a)$ where $r=\lceil
  b/(b-a)\rceil$. Let $\qq=(x^c,x^by^a,x^ay^b,y^c)R$ be an ideal in a
  power series ring $R:=k[[x,y]]$ over a field $k$. By
  \cite[Thm. 3.5]{HerzogQureshiSaem} we have
  \begin{equation*}
    \mathcal{F}_\qq(R)\simeq \frac{k[X,Y,U,V]}{(XU^{r-1},Y^{r-1}V,XV)k[X,Y,U,V]}.
  \end{equation*}
  Set $Q:=(XU^{r-1},Y^{r-1}V,XV)k[X,Y,U,V]$. Since
  $\sqrt{Q}=(X,U)\cap(Y,V)\cap(X,V)$, then
  $\operatorname{deg.rad}\mathcal{F}_\qq(R)=3$. Hence
  for any $l=3\cdot(4-2)+2=8$ independent linear forms
  $N_1,\ldots,N_8$ there exist $1\leqslant i_1<i_2\leqslant 8$ such
  that
  \begin{equation}\label{eq:3}
    (N_{i_1}(x^c,x^by^a,x^ay^b,y^c),N_{i_2}(x^c,x^by^a,x^ay^b,y^c))R
  \end{equation}
  gives a reduction of $\qq$.

  We show, that in fact for any
  $N_1,\ldots,N_6$ independent linear forms there exist $1\leqslant
  i_1<i_2\leqslant 6$ such that ideal \eqref{eq:3} is a reduction of
  $\qq$. By Claim \ref{claim} it is enough to prove that  $\height (Q
  + (N_{i_1},N_{i_2}))=4$ for some $1\leqslant i_1<i_2\leqslant
  6$. Assume to 
  the contrary, that for any choice of $1\leqslant i_1<i_2\leqslant
  6$ we have $\height (Q + (N_{i_1},N_{i_2}))<4$. First, let us assume,
  that for any such choice this height is equal to 2. Then, after a
  possible renumbering, we get
  \begin{equation*}
    \height (X, U, N_1, N_2) = \height (Y, V, N_3,
    N_4) = \height (X, V, N_5, N_6) = 2.
  \end{equation*}
  Hence $N_1,N_2,N_5,N_6\in (X,U,V)$ which gives a contradiction. Now,
  let us assume, that for any such choice the aforementioned height is
  equal to 3. Thus there exists, say $N_1$, such that
  \begin{equation*}
    \height (X, U, N_1) = \height (Y, V, N_1) = \height (X, V, N_1) = 3.
  \end{equation*}
  Then, again after a possible renumbering, we get
  \begin{equation*}
    \height (X, U, N_1, N_2, N_3) = \height (X, V, N_1, N_4) = 3,
  \end{equation*}
  or
  \begin{equation*}
    \height (Y, V, N_1, N_2, N_3) = \height (X, V, N_1, N_4) = 3.
  \end{equation*}
  In both cases we get a contradiction. Thus by Claim \ref{claim} we
  get the assertion. 
\end{example}

\begin{corollary}
  If all minimal prime ideals of
  $\mathcal{F}_\qq(R)$ are of the same height
  then in the Main Theorem and in Corollary \ref{cor:minimum_of_mult},
  $\operatorname{deg.rad}\mathcal{F}_\qq(R)$ may be replaced by 
  the multiplicity $e(\qq)$.
\end{corollary}

\begin{proof}
  By the assumption on $\mathcal{F}_\qq(R)$ we have
  $\operatorname{deg.rad}(\mathcal{F}_\qq(R))\leqslant\deg
  \mathcal{F}_\qq(R)$. 
  Moreover, by \cite[Proposition 2.4]{vasconcelos_2003} (see also the proof of
  \cite[Theorem 14.14]{matsumura}) we have
  $\deg\mathcal{F}_\qq(R)\leqslant e(\qq)$ hence the assertion.
\end{proof}


\bibliographystyle{abbrv}

\end{document}